\documentclass[12pt]{article}
\usepackage{amssymb,amsmath,amsthm,latexsym}
\usepackage[utf8]{inputenc}
\usepackage{geometry}
\usepackage{yfonts}
\usepackage{xfrac}
\usepackage[c]{esvect}
\usepackage{setspace}
\usepackage{graphicx}
\usepackage{indentfirst} 
\usepackage{amsmath,amsfonts,amsthm,amscd,upref,amstext}	
\usepackage[titletoc]{appendix}

\usepackage{color,soul}
\usepackage{xcolor}
\sethlcolor{yellow}

\usepackage{tablefootnote}
\usepackage{threeparttable}
\usepackage{booktabs}
\usepackage{siunitx,caption}
\usepackage{adjustbox}
\usepackage{tikz}
\usetikzlibrary{arrows.meta,automata,positioning,shadows}
\usepackage{amssymb,amsmath,amsthm,latexsym}
\usepackage{footnote}
\usepackage{graphicx}
\usepackage{float}
\usepackage{spverbatim}
\newtheorem{thm}{Theorem}[section]
\newtheorem{cor}[thm]{Corollary}
\newtheorem{lem}[thm]{Lemma}
\newtheorem{prop}[thm]{Proposition}

\newtheorem{rema}[thm]{Remark}
\newtheorem{defi}[thm]{Definition}

{ \theoremstyle{remark} }

\newcommand{\keyword}[1]{\textsf{#1}}

\title{A Classification of Elements of Function Space F($\mathbb{R}$,$\mathbb{R}$)}
\author{Mohsen Soltanifar\footnote{$^{1}$ \quad 
Analytics Division, College of Professional Studies, Northeastern University, 1400-410 West Georgia Street, Vancouver, BC V6B 1Z3, Canada \\
$^{2}$ \quad 
Biostatistics Division, Dalla Lana School of Public Health, University of Toronto, 620-155 College Street, Toronto, ON M5T 3M7, Canada\\
$^{3}$ \quad Biostatistics \& Programming Division, Biometrics Department, ClinChoice Inc, 750-2 Robert Speck Parkway, Mississauga, ON L4Z 1H8, Canada , E-mail: mohsen.soltanifar@alum.utoronto.ca}}

\date{\today}
\begin{document}
\maketitle
\doublespacing
\abstract{\noindent In this paper, we classify the function space of all real-valued functions on $\mathbb{R}$ denoted as $F(\mathbb{R},\mathbb{R})$ into 28 distinct blocks. Each block contains elements that share common features in terms of the cardinality of their sets of continuity and differentiability. Alongside this classification, we introduce the concept of the \emph{Connection}, which reveals a special relationship structure between four well-known real-valued functions in real analysis: the Cantor function, Dirichlet function, the Thomae function, and the Weierstrass function. Despite the significance of this field, several perspectives remain unexplored.}

\keyword{\textbf{Key Words:} Real-valued functions, Cardinals, Cantor function, Thomae function, Weierstrass function, Dirichlet function, Partition }


\begin{center}
\emph{From the paradise, that Cantor created for us, no-one shall be able to expel us.\newline\\
\hspace{10 cm}David Hilbert-1925}
\end{center}

\section{Introduction}
\subsection{Real Valued Functions}
The theory of functions of a real variable was first treated by the Italian mathematician Ulisse Dini in 1878 \cite{Medvedev1991}. This theory was constructed through the enlargement and deepening of set theory and was later developed separately and in parallel \cite{Medvedev1991}. The development of this theory can be divided into three periods: the first period (1867-1902) saw extensive investigation into various topics of classical analysis, such as integrals, derivatives, and point set theory; the second period (1902-1930) marked the solidification of the theory of functions of a real variable as an independent mathematical discipline; and, the third period (1930-present) is characterized by the study of the theory of real-valued functions in connection with functional analysis \cite{Medvedev1991}.\par 

\subsection{Motivation}
"The function space of all real-valued functions on the real line ($F(\mathbb{R},\mathbb{R})$) is an infinite-dimensional vector space with a nonconstructive basis. It encompasses various types of functions with pathological and chaotic structures. As researchers' attention has shifted from pure existential mathematics to constructive mathematics \cite{Waaldijk2005, Troelstra1988}, mathematicians have made numerous attempts to focus on special subsets of this vast vector space (e.g., all real-valued continuous functions \cite{Pugh2002}) or to classify this vector space using novel ancillary concepts. Examples of such classifications include those based on (i) set theory (surjective, injective, bijective, etc.), (ii) operators (additive, multiplicative, even, odd, etc.), (iii) topology (continuous, open, closed, etc.), (iv) properties concerning real numbers (differentiable, smooth, convex, etc.), and (v) measurability (Borel status, Baire status, etc.) \cite{Hairer2008, Stillwell2013, Royden2023}.\par 

This work presents another attempt to describe the vector space of real-valued functions on the real line in a constructive manner. Specifically, it aims to classify the elements of this vector space using their associated information on cardinality, continuity, and differentiability. Furthermore, based on this classification, it establishes a particular relationship among four of them: the Cantor function, the Dirichlet function, the Thomae function, and the Weierstrass function."\par   

\subsection{Study Outline}
This paper is divided into four sections. The first section provides the necessary preliminaries in set theory, linear algebra, and special functions, which are essential for the following sections. In the second section, we discuss the partition of $F(\mathbb{R},\mathbb{R})$ into 28 blocks of functions, presenting constructive examples for each block. The third section focuses on the relationship between these 28 blocks of functions using graph theory, with a specific emphasis on the four main plausible functions. Finally, we conclude the work with a discussion section on the current results and future directions.\par 

\section{Preliminaries}
Readers who have studied the key topics of Analysis and Linear Algebra are well-equipped with the following notations, definitions, and results in the areas of "Set Theory" \cite{LinLin1981, Soltanifar2006a, Soltanifar2006b}, "Linear Algebra" \cite{Lipshutz1981, LinLin1981, Ventre2023, Aron.etal.2004}, and "Special Functions" \cite{Dunham2018, Bass2013, Gelbaum2003, Bourchtein2014}.\par

\subsection{Set Theory}

\begin{prop}\label{prop.card} Let $\mathbb{R}$ denote the set of real numbers and $A\subseteq \mathbb{R}.$ Then, either $A$ is empty, non-empty finite, denumerable, or  uncountable. In these cases, we denote the cardinality of $A$ by $0, n(n\in\mathbb{N}), \aleph_0,$ or $c,$ respectively. 
\end{prop}
\begin{rema}
Henceforth, we assume there are only four types of subsets in the real line given Proposition \ref{prop.card}, and considering all non-empty finite sets of one category with symbol $n$ as their cardinal number. 
\end{rema}

\begin{prop}{\ \\}\label{prop.propcard}
(i) Let $A\subseteq\mathbb{R}$ be uncountable and $A=A_1\dot\cup A_2$ Then, $A_1$ or $A_2$ is uncountable.\\
(ii) Let $A\subseteq B \subseteq \mathbb{R}.$ Then, $0\leq Card(A)\leq Card(B)\leq c.$
\end{prop}

\begin{prop}\label{prop.propcantor}
Let $C$ be the ternary Cantor set, i.e $C=\{x\in [0,1] | x:=\sum_{n=1}^{\infty} \frac{a_n}{3^n}: a_n=0,2 (n\in \mathbb{N})  \}$. Then, one can write $C=\cap_{n=1}^{\infty} C_n$ where $C_n$ is the disjoint union of $2^n$ intervals of the form $I_{n,k}:=[a_{n,k},b_{n,k}]\ \ (1\leq k\leq n)$ each of the length $3^{-n}, (n\geq 1).$  
\end{prop}

\begin{rema}\label{CantorUncSubset}
Given $\mathbb{C}_{unc}=\{ A\subseteq C| Card(A)=c \}$ we have:  $Card(\mathbb{C}_{unc})=2^c.$
\end{rema}

\subsection{Linear Algebra }

\begin{defi}\label{def.funspace} Let $\mathbb{R}$ denote the set of real numbers. We define (i) $F(\mathbb{R},\mathbb{R})=\{ f:\mathbb{R}\rightarrow\mathbb{R} | f\ is\ a\ function.\};$ (ii) $C(\mathbb{R},\mathbb{R})=\{ f:\mathbb{R}\rightarrow\mathbb{R} | f\ is\ a\ continuous\  function\  everywhere.\}$; and,  
(iii) $D(\mathbb{R},\mathbb{R})=\{ f:\mathbb{R}\rightarrow\mathbb{R} | f\ is\ a\ differentiable\  function\  everywhere.\}$

\end{defi}

\begin{rema}
The function space $F(\mathbb{R},\mathbb{R})$ equipped with conventional addition $+$ and scalar multiplication $\bullet$ constitutes the vector space $(F(\mathbb{R},\mathbb{R}), \mathbb{R},+,\bullet)$ . In addition, the vector spaces $(C(\mathbb{R},\mathbb{R}), \mathbb{R},+,\bullet)$, and $(D(\mathbb{R},\mathbb{R}), \mathbb{R},+,\bullet)$ are its sub-spaces.    
\end{rema}

\begin{prop} \label{prop.funcspace.relat} 
 Given the sets introduced in Definition\ref{def.funspace} we have: (i)  $  D(\mathbb{R},\mathbb{R})\subsetneq C(\mathbb{R},\mathbb{R}) \subsetneq F(\mathbb{R},\mathbb{R});$
 (ii)  $Card(F(\mathbb{R},\mathbb{R}))=2^c;$
 (iii) $Card(C(\mathbb{R},\mathbb{R}))=c;$ and, (iv) $Card(D(\mathbb{R},\mathbb{R}))=c.$    
 \end{prop}

\begin{prop}\label{prop.funcspace.dim} 
$\dim(F(\mathbb{R},\mathbb{R}))=2^c.$
\end{prop}

\subsection{Special Functions}

\begin{defi}\label{mainfunctions}
Given indicator function  $1_A (A\subseteq\mathbb{R}),$ and the ternary Cantor set $C.$ Then, the Cantor function $C(.),$ the Dirichlet function $D(.),$ the Thomae function $T(.),$ and the Weierstrass function $W(.)$ are defined on closed unit interval as:\par 
\begin{eqnarray}
C(x:=\sum_{n=1}^{\infty}\frac{a_{n,x}}{3^n})&=&\frac{1}{2^{N_x}}+\sum_{n=1}^{N_x-1}\frac{a_{n,x}}{2^{n+1}}: N_x=\min \{n\in\mathbb{N}: a_{n,x}=1 \},\\
D(x)&=& 1_{\mathbb{Q}}(x),  \\
T(\frac{m}{n}1_{\mathbb{Q}}(x:=\frac{m}{n})+x1_{\mathbb{Q}^c}(x))&=&\frac{1}{n}1_{\mathbb{Q}}(x):   (m,n)=1,  \\
W(x)&=& \sum_{n=0}^{\infty}\frac{\cos(21^n \pi x)}{3^n}.
\end{eqnarray}
\end{defi}
\begin{rema}
We note that the definitions of the above functions have straightforward extension from the closed unit interval to the entire real line. Also, the introduced  Weierstrass function here is special case of the general form for $a=\frac{1}{3},\ b=21.$
\end{rema}

\begin{defi}\label{minorfunctions}
Let $C$ be the ternary Cantor set, $D(.)$ be the Dirichlet function, and, $-\infty< a<b<\infty. $ Then, for  the triangular function $T^{(1)}(.)$ given by $T_{a,b}^{(1)}(x)=\sqrt{3}\Big(\frac{b-a}{2}-|x-\frac{b+a}{2}|\Big)1_{[a,b]}(x),$  and the transformed cosine function $T^{(2)}(.)$  given by   $T_{a,b}^{(2)}(x)=(b-a)*(1-cos(2\pi(\frac{x-a}{b-a})))1_{[a,b]}(x)$ we define two functions $f_C$ and $g_C$ on the real line as:\par
\begin{eqnarray}
f_C(x)&=& (\sum_{n=1}^{\infty}f_n(x))D(x): f_n(x)=\sum_{k=1}^{2^n}T_{a_{n,k},b_{n,k}}^{(1)}(x)1_{ [a_{n,k},b_{n,k}] }(x) \ (n\geq 1)\\
g_C(x)&=& (\sum_{n=1}^{\infty}g_n(x))D(x): g_n(x)=\sum_{k=1}^{2^n}T_{a_{n,k},b_{n,k}}^{(2)}(x)1_{ [a_{n,k},b_{n,k}]}(x)  \ (n\geq 1) .
\end{eqnarray}
\end{defi}
\begin{rema}
While the triangular function is continuous everywhere and has no derivative at points $x=a,b,$ the linear transformed cosine function is differentiable everywhere and in particular it derivative at points $x=a,b$ is zero. These properties will be inherited by the associated functions $f_C$ and $g_C,$ respectively. 
\end{rema}

\section{Main Results}
\subsection{Partition of $F(\mathbb{R},\mathbb{R})$ with Scenario Classification \& Examples}

\begin{thm}\label{theorem.partition}
The function space $F(\mathbb{R},\mathbb{R})$ may be partitioned into 28 unique distinct blocks of functions: 
\begin{eqnarray}
F(\mathbb{R},\mathbb{R}) &=& \dot\cup_{i=1}^{28}  [f_i] 
\end{eqnarray}
where in which 

\begin{eqnarray}
[f_i] &=& \{ f\in F(\mathbb{R},\mathbb{R})  | Card(C_f)=Card(C_{f_i}), Card(D_f)=Card(D_{f_i}) \}  (1\leq i\leq 28),\\
C_f &=& \text{the set of continuity points of the function f},\\
D_f &=& \text{the set of differentiability points of the function f}.
\end{eqnarray}
\end{thm}

\begin{proof} 
First, let $f\in F(\mathbb{R},\mathbb{R}),$ and consider its set of continuity points $C_f$. Then, given $\mathbb{R}=C_f \dot\cup C_f^c,$ by an application of Proposition \ref{prop.propcard} (i) it follows that at least $C_f$ of $C_f^c$ in uncountable. Subsequently, by Proposition \ref{prop.propcard} (ii) there are seven different scenarios for the cardinality of the pair $(C_f,C_f^c)$ including $(0,c),(n,c),(\aleph_0,c),(c,c),(c,\aleph_0),(c,n),(c,0).$ A similar argument for the set of differentiabilities $D_f$ with seven blocks holds. Second, by multiplication principle, it appears that there are $7\times 7=49$ blocks of functions. However, by two applications of Proposition \ref{prop.propcard} (ii) for $A = D_f$ and $B = C_f,$ and, for $A = C_f^c$ and $B = D_f^c,$ only $(1+2+\cdots+7) = 28 $ blocks exist. Table \ref{Table2} enlists these blocks.\par 
\end{proof}

\begin{cor}
Each of the blocks of functions with most chaotic structure ($[f_1]$) and functions of the least chaotic structure ($[f_{28}]=D(\mathbb{R},\mathbb{R})$) constitutes only 3.5\% (1/28) of all blocks. Hence, 93\% of blocks of functions fall between these two opposite extremes. Furthermore, the vector space of everywhere continuous functions on the real line ($C(\mathbb{R},\mathbb{R})$) constitutes 25\% (7/28) of all blocks.
\end{cor}

\begin{thm}\label{thm.examples}
There is at least one constructive example representing each of 28 unique distinct blocks of functions in $F(\mathbb{R},\mathbb{R}).$    
\end{thm}
\begin{proof} We provide the representative functions as listed in Table \ref{Table3}.
\end{proof}

\begin{center}
\begin{table}[H]
\caption{List of 28 representatives  blocks of partition of $F(\mathbb{R},\mathbb{R})$\label{Table2}}
\begin{threeparttable}[b] 
\centering 
\begin{adjustbox}{width=0.55\textwidth, height=0.55\textwidth}
\begin{tabular}{ll cc lcc} 
\hline 
&& \multicolumn{2}{c}{Continuity} & & \multicolumn{2}{c}{Differentiability}\\ [0.5ex]
\hline 
\#&Case&   $Card(C_f)$ & $Card(C_f^c)$&  &   $Card(D_f)$ & $Card(D_f^c)$\\[1ex] 
\hline 
1&1-1&     $0$ & $c$ &  &     $0$ & $c$\\[1ex] 
2&2-1&     $n$ & $c$ &  &     $0$ & $c$\\[1ex] 
3&2-2 &    $-$ & $-$&   &     $n$ & $c$  \\[1ex] 
4&3-1&     $\aleph_0$ & $c$ &  &     $0$ & $c$\\[1ex] 
5&3-2 &    $-$ & $-$&   &     $n$ & $c$  \\[1ex] 
6&3-3 &   $-$ & $-$ &  &    $\aleph_0$ & $c$\\[1ex]
7&4-1&     $c$ & $c$ &  &     $0$ & $c$\\[1ex] 
8&4-2 &    $-$ & $-$&  &     $n$ & $c$  \\[1ex] 
9&4-3 &   $-$ & $-$ & &    $\aleph_0$ & $c$\\[1ex]
10&4-4 &    $-$ & $-$ & &     $c$ & $c$\\[1ex]  
11&5-1&     $c$ & $\aleph_0$ & &     $0$ & $c$\\[1ex] 
12&5-2 &    $-$ & $-$&   &     $n$ & $c$  \\[1ex] 
13&5-3 &   $-$ & $-$ &  &    $\aleph_0$ & $c$\\[1ex]
14&5-4 &    $-$ & $-$ &  &     $c$ & $c$\\[1ex]  
15&5-5 &    $-$ & $-$ &  &  $c$ & $\aleph_0$\\[1ex] 
16&6-1&     $c$ & $n$ &  &     $0$ & $c$\\[1ex] 
17&6-2 &    $-$ & $-$&   &     $n$ & $c$  \\[1ex] 
18&6-3 &   $-$ & $-$ & &    $\aleph_0$ & $c$\\[1ex]
19&6-4&    $-$ & $-$ &  &     $c$ & $c$\\[1ex]  
20&6-5 &    $-$ & $-$ &  &  $c$ & $\aleph_0$\\[1ex] 
21&6-6 &   $-$ & $-$ &  &    $c$ & $n$\\[1ex] 
22&7-1&     $c$ & $0$ &  &     $0$ & $c$\\[1ex] 
23&7-2 &    $-$ & $-$&   &     $n$ & $c$  \\[1ex] 
24&7-3 &   $-$ & $-$ &  &    $\aleph_0$ & $c$\\[1ex]
25&7-4 &    $-$ & $-$ &  &     $c$ & $c$\\[1ex]  
26&7-5 &    $-$ & $-$ &  &  $c$ & $\aleph_0$\\[1ex] 
27&7-6 &   $-$ & $-$ & &    $c$ & $n$\\[1ex] 
28&7-7 &    $-$ & $-$ &  &     $c$ & $0$\\[1ex] 
\hline 
\end{tabular}
 \end{adjustbox} 
\end{threeparttable}
\end{table}
\end{center} 
\begin{center}
\begin{table}[H]
\caption{List of 28 representatives  functions for each of 28 blocks  $F(\mathbb{R},\mathbb{R})$\label{Table3}}
\begin{threeparttable}[b] 
\centering 
\begin{adjustbox}{width=0.55\textwidth, height=0.55\textwidth}
\begin{tabular}{ll lll} 
\hline 
\#&Case&   Representative $f(x)$ &  Comments & $Card([f])$\\[1ex] 
\hline 
1&1-1& $D(x)$  &  Dirichlet Function &  $2^c$ \\[1ex] 
2&2-1& $(\prod_{k=1}^{n}(x-k))D(x)$  &  & $2^c$\\[1ex] 
3&2-2 & $(\prod_{k=1}^{n}(x-k)^2)D(x)$ &  & $2^c$ \\[1ex] 
4&3-1& $\sin(\pi x) D(x)$ &  & $2^c$ \\[1ex] 
5&3-2 & $(\sin(\pi x)\prod_{k=1}^{n}(x-k)) D(x)$ & &$2^c$\\[1ex] 
6&3-3 & $(\sin^2(\pi x)) D(x)$ & & $2^c$\\[1ex]
7&4-1& $f_C(x)$  &  $C_f = C$ & $2^c$\\[1ex] 
8&4-2 & $(\prod_{k=1}^{n}(x-\frac{1}{3^k})^2)f_C(x)$ & & $2^c$ \\[1ex] 
9&4-3 & $(\sin^2(\frac{\pi}{x}))f_C(x)$ &  & $2^c$\\[1ex]
10&4-4 & $g_C(x)$ &  $C_f = D_f =  C$ & $2^c$\\[1ex]  
11&5-1& $T(x)$  &  Thomae Function& $c$\\[1ex] 
12&5-2 & $(\prod_{k=1}^{n}(x-k)^2)T(x)$ &  & $c$  \\[1ex] 
13&5-3 & $(\sin^2(\pi x))T(x)$ & & $c$\\[1ex]
14&5-4 & $T(x)1_{[0,1]}(x)$ & & $ c$ \\[1ex]  
15&5-5 & $\sum_{n=1}^{+\infty}\pi^n1_{\{\pi^n\}}(x)$ & & $c$ \\[1ex] 
16&6-1& $W(x)+\sum_{k=1}^{n}\pi^k1_{\{\pi^k\}}(x)$  & & $c$ \\[1ex] 
17&6-2 & $(\prod_{k=1}^{n}(x-k)^2)(W(x)+\sum_{k=1}^{n}\pi^k1_{\{\pi^k\}}(x))$ &   &$c$\\[1ex] 
18&6-3 &  $(\sin^2(\pi x))(W(x)+\sum_{k=1}^{n}\pi^k1_{\{\pi^k\}}(x))$ & & $c$\\[1ex]
19&6-4&  $(W(x)+\sum_{k=1}^{n}\pi^{-k}1_{\{\pi^{-k}\}}(x))1_{[0,1]}(x)$ & & $c$ \\[1ex]  
20&6-5 & $|\sin(\pi x)|+\sum_{k=1}^{n}\pi^k1_{\{\pi^k\}}(x)$ & & $c$  \\[1ex] 
21&6-6 & $\sum_{k=1}^{n}\pi^k1_{\{\pi^k\}}(x)$ & & $c$\\[1ex] 
22&7-1 & $ W(x)$ &  Weierstrass Function &$c$ \\[1ex] 
23&7-2 & $(\prod_{k=1}^{n}(x-k)^2)W(x)$ &   & $c$ \\[1ex] 
24&7-3 & $(\sin^2(\pi x))W(x)$ & & $c$\\[1ex]
25&7-4 & $C(x)$ & Cantor Function & $c$\\[1ex]  
26&7-5 & $|\sin(\pi x)|$ & &$c$  \\[1ex] 
27&7-6 &$|\prod_{k=1}^{n}(x-k)|$ & &$c$ \\[1ex] 
28&7-7 & $x$ &  &$c$\\[1ex] 
\hline 
\end{tabular}
\end{adjustbox} 
\end{threeparttable}
\end{table}
\end{center}

\begin{rema}
The block of functions with least chaotic structure ($[f_{28}]$) represented by the identity function $I$ in the Table \ref{Table3} enlists many well known functions including all polynomials $p_m(.) (m\geq 0)$, the trigonometric functions $sin(.),cos(.),$ the exponential function $\exp(.),$ and the Volterra's function \cite{Bressoud2008}.    
\end{rema}

\begin{rema}
The ternary Cantor set $C$ appears in the construction of representatives of $17.9\%(5/28)$ of the blocks. This shows that the ternary Cantor set has remarkable presence in the representative blocks of the space of real valued functions on the real line.
\end{rema}

\begin{rema}
The four well-known functions$C(.),D(.),T(.),$ and $W(.)$ show up in  $64.3\%(18/28)$ of the blocks where for given representative function $f_i= (h_{i1}+h_{i2})h_{i3}$ at least one of $h_{ij}  (j=1,2,3)$ is one of these four functions. \end{rema}

\begin{lem}\label{card.blockunion}
Let $F_1(\mathbb{R},\mathbb{R})=\cup_{i=1}^{10} [f_i]$ and  $F_2(\mathbb{R},\mathbb{R})=\cup_{i=11}^{28} [f_i]$. Then:\newline 
(i) $Card(F_1(\mathbb{R},\mathbb{R}))=2^c$\newline 
(ii) $Card(F_2(\mathbb{R},\mathbb{R}))=c.$   \end{lem}
\begin{proof}
First, given $F(\mathbb{R},\mathbb{R})=F_1(\mathbb{R},\mathbb{R})\dot\cup F_2(\mathbb{R},\mathbb{R}),$ we have $2^c=Card(F_1(\mathbb{R},\mathbb{R}))+Card(F_2(\mathbb{R},\mathbb{R})).$ Hence, proving claim (ii) yields proving claim (i).
Second, to prove claim (ii) let $X_{count}\subset \mathbb{R}$ be countable. Then, given $Card(C(X_{count}^c,\mathbb{R}))=c$ \cite{Hodel1984} we have $Card(\cup_{X_{count}\subset \mathbb{R}}C(X_{count}^c,\mathbb{R}))=c.$ Consequently, 
\begin{equation}\label{prod.cardinal}
Card(F(\mathbb{N},\mathbb{R})\times \cup_{X_{count}\subset \mathbb{R}}C(X_{count}^c,\mathbb{R}))=c.     
\end{equation}
Finally, using the result in equation \ref{prod.cardinal} it is sufficient to consider the mapping $F_2(\mathbb{R},\mathbb{R}) \mapsto (F(\mathbb{N},\mathbb{R})\times \cup_{X_{count}\subset \mathbb{R}}C(X_{count}^c,\mathbb{R}))$ which maps  $f\in F_2(\mathbb{R},\mathbb{R})$ to $(f|_{C_f^c},f|_{C_f}).$\par  This complete the proof.
\end{proof}

\begin{thm}\label{thm.blockcard}
\begin{eqnarray}
Card([f_i])&=& 1_{[1,10]}(i)\times 2^c+1_{[11,28]}(i)\times c\ \ \ (1\leq i\leq 28).
\end{eqnarray}  
\end{thm}
\begin{proof}
\indent First, let $f \in F(\mathbb{Q}^c,\mathbb{R}_{0}^{-})$. Then, we can extend $f$ to $\mathbb{R}$ by $f(\mathbb{Q})=\{1\}$. A straightforward verification shows that $f \in [f_1]$. Thus,  given the mapping $F(\mathbb{Q}^c,\mathbb{R}_{0}^{-}) \mapsto [f_1]$ we have $2^c=Card(F(\mathbb{Q}^c,\mathbb{R}_{0}^{-}))\leq Card([f_1]) \leq Card (F(\mathbb{R},\mathbb{R}))=2^c$. This yields,  $Card([f_1])=2^c$. Next, let $f\in [f_1]$ and define $g_i(.)\ (2\leq i\leq 6)$  by $g_i(x)=(\prod_{k=1}^{n}(x-k)), (\prod_{k=1}^{n}(x-k)^2),\sin(\pi x), (\sin(\pi x)\prod_{k=1}^{n}(x-k)), (\sin^2(\pi x))$ for $(2\leq i\leq 6),$ respectively. Then, using the CDF of the normal distribution $\Phi,$ we have $g_i\Phi(f)\in [f_i]$ for  $(2\leq i\leq 6),$ respectively. Thus,  given the mapping $ [f_1] \mapsto [f_i]$ we have $2^c=Card([f_1])\leq Card([f_i])\leq Card (F(\mathbb{R},\mathbb{R}))=2^c,$ implying $Card([f_i])=2^c$ for  $(2\leq i\leq 6),$ respectively. \par 
\indent Second, let $A\in \mathbb{C}_{unc}$ and consider the modified $f_i^A(.)\ (7\leq i\leq 10)$ in Table \ref{Table3} by $f_i^A(x)=f_A(x), (\prod_{k=1}^{n}(x-1/3^k))f_A(x), (\sin^2(\frac{\pi}{x}))f_A(x), g_A(x)$ for $(7\leq i\leq 10),$ respectively.
 Thus,  given the mapping $ \mathbb{C}_{unc} \mapsto [f_i]$ we have $2^c=Card(\mathbb{C}_{unc})\leq Card([f_i])\leq Card (F(\mathbb{R},\mathbb{R}))=2^c,$ implying $Card([f_i])=2^c$ for  $(7\leq i\leq 10),$ respectively.  \par

\indent Third, let $ p\in \mathbb{R}^{+}.$ Then, using function $f_i\ (11\leq i\leq 28)$ in the Table\ref{Table3} we have $pf_i\in [f_i]\ (11\leq i\leq 28).$ Hence, given the mapping $ \mathbb{R}^{+}  \mapsto [f_i]$ we have $c=Card(\mathbb{R}^{+})\leq Card([f_i]),$ for  $(11\leq i\leq 28),$ respectively. Next, given $[f_i]\subseteq F_2(\mathbb{R},\mathbb{R}),$  for  $(11\leq i\leq 28),$ and  Lemma\ref{card.blockunion} (ii) we have: $Card([f_i])\leq c$ for  $(11\leq i\leq 28),$ respectively.
Accordingly, by the last two inequalities on the cardinals, it follows that: $Card([f_i])=c \ (11\leq i\leq 28).$ In particular, for $[f_{28}]=D(\mathbb{R},\mathbb{R})$ we have  $Card([f_{28}])=c.$ \par
This completes the proof.
\end{proof}

\begin{rema}
The cardinality of $35.7\%(10/28)$  of the blocks is $2^c$ while that of  $64.3\%(18/28)$   of  blocks is $c.$ This indicates that the cardinal number $c$ has almost double frequency of that of cardinal number $2^c$ in representing the size of blocks of the space of real valued functions on the real line.
\end{rema}

\subsection{The Relationship between the Big Four}

In the previous section, we observed that any given function $f\in F(\mathbb{R},\mathbb{R})$ belongs to one of the 28 blocks of its partitions. In particular, the four well-known functions—the Cantor function, Dirichlet function, Thomae function, and the Weierstrass function—each represent one of these blocks. Now, one may wonder how to connect these four key functions. Trivially, the equivalence relation induced by the aforementioned partition is unhelpful in this regard. Hence, we consider an alternative approach. We begin with a definition:\par 
\begin{defi}\label{defconn}
Given two functions $f_1,f_2\in F(\mathbb{R},\mathbb{R}).$ Then, $f_1$ is called connected to $f_2$ denoted by $f_1\overset{conn}{\sim}f_2$ whenever for some function $g\in F(\mathbb{R},\mathbb{R}),$ we have $f_1g\in[f_2].$
\end{defi}

\begin{rema}\label{rema.conn}
Additional restrictions on $g$: When $g=1$ in Definition \ref{defconn}, $\overset{conn}{\sim}$ is transformed into the induced equivalence relation by the above partition. Furthermore, if $g$ is a positive everywhere-differentiable function on $\mathbb{R}$, the aforementioned relation becomes an equivalence relation.
\end{rema}

It is trivial that in Theorem \ref{theorem.partition}, $f_i\overset{conn}{\sim}f_{28}$ for all $(1\leq i\leq 27).$ Figure \ref{FIG1} presents these relationships. As shown, the block of everywhere differentiable functions $[f_{28}]$ is the sink node with the highest in-degree connectivity among all blocks of functions.\par 

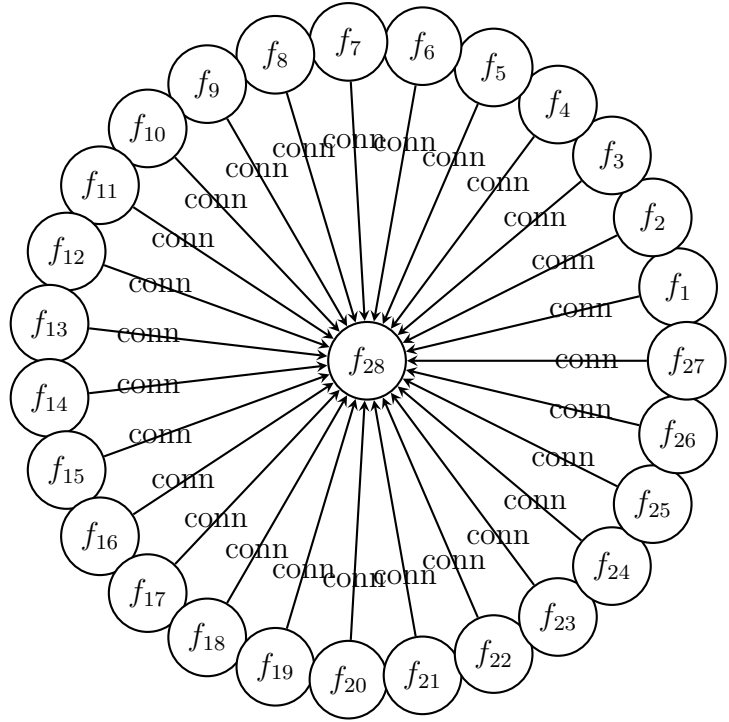
\begin{figure}[H]
\centering
\begin{tikzpicture}[-stealth,thick,scale=1, level distance=8em]
    \node[state,fill=white] (center) at (0,0) {$f_{28}$};
\foreach \phi in {1,...,27}{
\node[state,fill=white] (f_\phi) at (360/27 * \phi:4.25cm){$f_{\phi}$};
         \draw[black] (f_\phi) ->  node[near start]{conn}(center);
    }
\end{tikzpicture} 
\caption{The graphical presentation of the relationship between all representatives of blocks of $F(\mathbb{R},\mathbb{R})$ and $f_{28}$ the identity function.} \label{FIG1}
\end{figure}

As there are $4^{C_{2}^{28}}=3.790327\times 10^{227}$ potential scenarios for the complete graph in Figure \ref{FIG1}, finding the relationships between all nodes of the graph appears to be a tedious and difficult task. However, we can identify the relationships between four of them, i.e., $f_{1}=D, f_{11}=T, f_{22}=W,$ and $f_{25}=C,$ where there are only $4^{C_{2}^{4}}=4096$ potential scenarios. Equipped with Definition \ref{defconn}, we have:\par 

\begin{thm}\label{thmgraph}
Given above notations and definitions we have: 
(i) $W\overset{conn}{\sim}T,$  (ii) $W\overset{conn}{\sim}C,$ (iii) $W\overset{conn}{\sim}D,$ (iv) $T\overset{conn}{\sim}D,$  and, (v) $C\overset{conn}{\sim}D.$ 
\end{thm}
\begin{proof}
It is sufficient for each case to present the $g$ function in the Definition \ref{defconn} as follows: (i) $g(x)=x+\sum_{n=1}^{+\infty} \pi^n 1_{\{\pi^n\}}(x)$;  (ii) $g(x)=x(x-1)1_{[0,1]}(x)$; (iii)-(v) $g(x)=\sum_{-\infty}^{+\infty} 1_{A+2n}(x): A=(\sin(n))_{n=1}^{+\infty}$ dense in $[-1,1]$.    
\end{proof}

Figure \ref{FIG2} presents a graphical overview of the results in Theorem \ref{thmgraph}. As it is shown, the Weierstrass function (W) is the source-universal node with the highest out-degree connectivity; the Cantor function (C) and the Thomae function (T) are the bridging nodes; and, the Dirichlet function (D) is the sink node with the highest in-degree connectivity.\par    

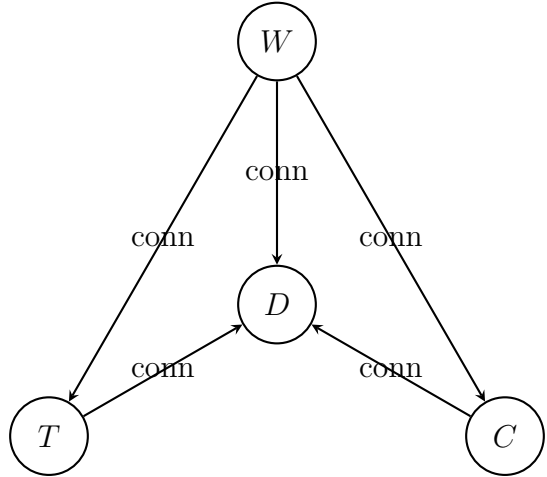
\begin{figure}[H]
\centering
\begin{tikzpicture}[-stealth,thick,scale=1,rotate=-30, level distance=8em]
    \node[state,fill=white] (center) at (0,0) {$D$};
\node[state,fill=white] (W) at (360/3 * 1:3.5cm){$W$};
         \draw[black] (W) -> node{conn}(center);
\node[state,fill=white] (T) at (360/3 * 2:3.5cm){$T$};
         \draw[black] (T) -> node{conn}(center);
\node[state,fill=white] (C) at (360/3 * 3:3.5cm){$C$};
         \draw[black] (C) -> node{conn}(center);
\draw[black] (W) -> node{conn}(T);
\draw[black] (W) -> node{conn}(C);         
\end{tikzpicture}

\caption{The graphical presentation of the relationship between well-known functions: the Cantor function (C), the Dirichlet function (D), the Thomae function (T), and, the Weierstrass function (W).} \label{FIG2}
\end{figure}

\section{Discussion}

\subsection{Summary \& Contributions}
This work presents a finite partition of the function spaces of all real-valued functions on $\mathbb{R}$ based on cardinality, continuity, and differentiability, along with constructive examples representing each block of the partition. In particular, it shows that the well-known Cantor function, Dirichlet function, Thomae function, and the Weierstrass function each represent a unique block of this partition. An additional aspect adding more importance to these four functions is that they collectively appear in representation of almost two-thirds of the blocks. Furthermore, the concept of \emph{connection} among real-valued functions is introduced, and the unique connection relation between the aforementioned functions was investigated.\par 
Finally, this work findings adds more prominence to the Cantor set $C$ as well. While it has a remarkable presence in the construction of the representative functions of blocks, its size (e.g., Cardinal number $c$) has the highest presence in the size of representative blocks.\par 

\subsection{Limitations \& Future Work}
The limitations of this work are clear, and they open up new perspectives for further investigations. Firstly, we merged the cardinal number of all finite subsets of $\mathbb{R}$ with given symbol $n.$ While this inaccuracy is a minimal price to pay for enabling the creation of the aforementioned finite partition, it should be noted. Secondly, the presented graph in Figure \ref{FIG1} needs to be completed for all its involved nodes. Thirdly, the \emph{connection} relation in Definition \ref{defconn} is not an equivalence, making it suboptimal. One open problem in this regard is investigating the results in Figure \ref{FIG1} and Figure \ref{FIG2} when considering the equivalence relations mentioned in Remark \ref{rema.conn}. Finally, it is worth exploring how the presented partition in this work changes in terms of the number of blocks and the representative function for each block when one replaces some key features, such as continuity and differentiability, with other properties of real-valued functions on the real line, such as integrability, measurability, etc.\par

\subsection{Conclusion}
This work presented a constructive description of the function space of all real-valued functions on $\mathbb{R}$ in terms of four concepts: partition, cardinality, continuity, and differentiability. Additionally, it established a special relationship between the well-known representative functions of four of the blocks.\par



\begin{thebibliography}{99}


\bibitem{Medvedev1991}
Medvedev, F. A. \textit{Scenes from the History of Real Functions}, 1st ed.; In Birkhäuser Basel eBooks. Springer Basel AG, Switzerland, 1991; pp. 11--14. 


\bibitem{Waaldijk2005}
Waaldijk, F. On the Foundations of Constructive Mathematics–Especially in Relation to the Theory of Continuous Functions. {\em Found. Sci.} {\bf 2005}, {\em 10}, 249–324.



\bibitem{Troelstra1988}
Troelstra, A.S.; van Dalen, D. \textit{Constructivism in Mathematics: An Introduction}, 1st ed.; (Two Volumes); Elsevier Science, Amsterdam: North Holland, The Netherlands, 1988.

\bibitem{Pugh2002}
Pugh, C. C. \textit{Real Mathematical Analysis}, 1st ed.;In Undergraduate texts in mathematics. Springer Science Business Media, New York, USA, 2002; pp. 223--225.


\bibitem{Hairer2008}
Hairer, E.; Wanner, G. \textit{Analysis by Its History}, 1st ed.; 
Springer Science+Business Media, LLC, NY, USA, 2008. 


\bibitem{Stillwell2013}
Stillwell, J.  \textit{The real numbers: An Introduction to Set Theory and Analysis}, 1st ed.; Springer. Switzerland, 2013.


\bibitem{Royden2023}
Royden, H. L.; Fitzpatrick, P. \textit{Real Analysis}. 5th ed, Pearson, USA, 2023. 


\bibitem{LinLin1981}
Lin, S.T. and, Lin, Y.  \textit{Set Theory with Applications}, 2ed ed. Manner Publishing Company Inc., Tampa, FL, USA, 1981; pp. 147--149.


\bibitem{Soltanifar2006a}
Soltanifar, M. On A Sequence of Cantor Fractals. {\em Rose-Hulman Undergraduate Mathematics Journal} {\bf 2006}, {\em 7(1)}, Article 9.


\bibitem{Soltanifar2006b}
Soltanifar, M. A Different Description of A Family of Middle-a Cantor Sets. {\em American Journal of Undergraduate Research} {\bf 2006}, {\em 5(2)}, 9--12.

\bibitem{Lipshutz1981} 
Lipshutz, S.  \textit{Schaum's Ouline of Theory and Problems of Linear Algebra}, SI ed. McGraw Hill International, Singapore, 1981; pp. 64,83.


\bibitem{Ventre2023} 
Ventre, A.G.S.  \textit{Calculus and Linear Algebra: Fundamentals and Applications}. Springer, Switzerland, 2023; p.338.


\bibitem{Aron.etal.2004}
Aron, R. M.; Gurariy, V. I., \& Seoane, J. Lineability and spaceability of sets of functions on $\mathbb {R}$. {\em Proceedings of the American Mathematical Society} {\bf 2004}, {\em 133(3)}, 795-–803. https://doi.org/10.1090/s0002-9939-04-07533-1




\bibitem{Dunham2018}
Dunham, W.  \textit{The Calculus Gallery: Masterpieces from Newton to Lebesgue}. Princeton University Press, NJ, USA, 2018; pp.100,142,149


\bibitem{Bass2013}
Bass, Richard Franklin. \textit{Real analysis for Graduate Students}, 2ed ed. Createspace Independent Publishing, USA, 2013; pp.28,29.


\bibitem{Gelbaum2003}
Gelbaum, B. R.; Olmsted, J. M. H.  \textit{Counterexamples in Analysis}. Courier Corporation. Dover Publications, Inc. Mineola, New York, USA, 2003; pp. 22,27,38--39.


\bibitem{Bourchtein2014}
Bourchtein, A.; Bourchtein, L. \textit{Counterexamples: From Elementary Calculus to the Beginnings of Analysis}. CRC Press, Boca Raton, FL, USA, 2014; pp. 37,45,65.

\bibitem{Bressoud2008}
Bressoud, D.\textit{A Radical Approach to Lebesgue's Theory of Integration}. Cambridge: Cambridge University Press/Mathematical Association of America, USA, 2008; pp. 91--94.


\bibitem{Hodel1984}
Hodel, R. \textit{"Chapter 1: Cardinal Functions I." Handbook of Set-Theoretic Topology.} Elsevier
Science Publishers B.V., USA, 1984; pp.39--40



\end{thebibliography}
\end{document}